\documentclass{amsart}

\usepackage{amssymb}

\newcommand{\dist}{{\rm{dist}}}

\newcommand{\reals}{{\mathbb{R}}}

\newcommand{\bfr}{{\mathbf{r}}}
\newcommand{\bfx}{{\mathbf{x}}}

\newcommand{\olt}{{{\vec{t}}}}
\newcommand{\olu}{{{\vec{u}}}}

\newcommand{\scriptt}{{\mathcal{T}}}
\newcommand{\scriptf}{{\mathcal{F}}}
\newcommand{\scriptg}{{\mathcal{G}}}

\newcommand{\qtrcm}{{\hspace{.25cm}}}

\newcommand{\Lp}[2]{\|{#1}\|_{L^{#2}}}
\newcommand{\eps}{{\varepsilon}}

\newtheorem{thm}{Theorem}[section]
\newtheorem{cor}[thm]{Corollary}
\newtheorem{prop}[thm]{Proposition}
\newtheorem{lem}[thm]{Lemma}

\theoremstyle{remark}

\theoremstyle{definition}

\numberwithin{equation}{section}
\numberwithin{thm}{section}

\begin{document}


\title[Quasi-extremals]{Quasi-extremals for convolution with surface measure on the sphere}

\author{Betsy Stovall}
\address{Department of Mathematics\\
970 Evans Hall \#3840\\
University of California\\
Berkeley, CA 94720-3840
}

\email{betsy@math.berkeley.edu}

\thanks{The author was supported in part by NSF grant DMS-040126.}

\begin{abstract}
If $T$ is the operator given by convolution with surface measure on the sphere, $(E,F)$ is a quasi-extremal pair of sets for $T$ if $\langle T\chi_E, \chi_F \rangle \gtrsim |E|^{d/(d+1)}|F|^{d/(d+1)}$.  In this article, we explicitly define a family $\mathcal{F}$ of quasi-extremal pairs of sets for $T$.  We prove that $\mathcal{F}$ is fundamental in the sense that every quasi-extremal pair $(E,F)$ is comparable (in a rather strong sense) to a pair from $\mathcal{F}$.  This extends work carried out by M. Christ for convolution with surface measure on the paraboloid.
\end{abstract}

\maketitle

\section{Introduction}

Let $T$ be the linear operator which acts on the continuous functions
on $\reals^d$ by convolution with surface measure on the unit sphere
(which we will denote by $S^{d-1}$).  That is, for a continuous
function $f$ on $\reals^d$, $Tf$ is defined by 
\[
     Tf(x) = \int_{S^{d-1}}f(x-\omega)d\sigma(\omega).
\]

Then it is well known that $T$ extends to a continuous operator from
$L^p(\reals^d)$ to $L^q(\reals^d)$ if and only if $(p^{-1},q^{-1})$
lies in the closed triangle with vertices $(0,0)$, $(1,1)$, $(\frac{d}{d+1},
\frac{1}{d+1})$.  Our goal in this article will be to study the behavior of this
operator at the endpoint $(p,q) = (\frac{d+1}{d}, d+1)$ in more detail.  In
particular, we continue work begun by Christ in {\cite{QEx}} by
partially characterizing ``quasi-extremal'' and ``$\eps$-quasi-extremal'' pairs
for $T$.

Let $f$ and $g$ be measurable functions, not identically zero.  Then
by the boundedness of $T$ from $L^{(d+1)/d}$ to $L^{d+1}$ and duality,
we have that
  \[
   | \langle Tf, g \rangle |\lesssim \Lp{f}{(d+1)/d} \Lp{g}{(d+1)/d}.
  \]
We say that $(f,g)$ is an {\it $\eps$-quasi-extremal} pair if we have,
in addition, the lower bound
  \[
   | \langle Tf, g \rangle| \geq \eps \Lp{f}{(d+1)/d} \Lp{g}{(d+1)/d}.
  \]
We say that $(f,g)$ is simply {\it quasi-extremal} if $(f,g)$ is
$\eps$-quasi-extremal for some $\eps \gtrsim 1$.  

If $E$ and $F$ are Borel sets having positive Lebesgue measures, then
$(E,F)$ is an $\eps$-quasi-extremal or a quasi-extremal pair if
$(\chi_{E}, \chi_{F})$ is.

\subsection*{Notation}  As indicated above, we will write
\[
  S^{d-1} := \{ x \in \reals^d:|x| = 1\}.
\]
We will employ the symbols $\lesssim$, $\ll$, and their related symbols
$\gtrsim$, $\sim$, and $\gg$ as follows: Let $A$ and $B$ be positive
numbers and $P$ be some statement.  We will say that $P$ implies that $A
\lesssim B$ when there exists a (large) universal constant $C$ such that
$P$ implies that $A \leq CB$.  We will say that $A \ll B$ implies $P$
when there exists a (small) universal constant $c$ such that $A \leq cB$
implies that $P$ holds.  This use is fairly standard in the harmonic analysis literature.  We will also use
the somewhat less standard notation $\scriptt(E,F) := \langle T\chi_E,\chi_F\rangle$, for
measurable sets $E$ and $F$.  Here, $\langle g,f \rangle$ is the $L^2$ inner product, $\int_{\reals^d} f \overline{g} \, dx$, where $dx$ denotes Lebesgue measure.  Finally, we will write $|\cdot|$ to
indicate Lebesgue measure of a subset $E$ of $\reals^n$, where $n$ will be clear from the context.

\subsection*{Acknowledgements}  The author is indebted to her advisor Michael Christ for suggesting this question and for the invaluable help and advice he gave during work on this project.  She would also like to thank Z.~Gautam and T.~Tao for pointing out an error in an earlier formulation of Theorem \ref{mainthm}.  Finally, the author would like to express her gratitude toward C.~Thiele and the anonymous referee from the Illinois Journal of Math.\ for other helpful advice and suggestions.

\section {Statement of Results} \label{prelimrmks}

Before stating our results, we give an example of a quasi-extremal pair of sets.  Let $\rho \leq 1$ be a positive number, and let $\mathbf{r} =
(r_1,\ldots,r_{d-1})$ be a $(d-1)$-tuple of positive numbers satisfying
\begin{align}
    \label{rismall} \rho &\leq r_i \leq 1, \hspace{.5cm} 1 \leq i \leq d-1,\\
    \label{riinc} r_i &\leq r_{i+1}, \hspace{.5cm} 1 \leq i \leq d-2,\\
    \label{rinondeg} r_i &\geq \rho^{1/2} r_j, \quad 1 \leq i,j \leq d-1.
\end{align}
Note that (\ref{rinondeg}) is equivalent to $r_1 \geq \rho^{1/2} r_{d-1}$ on account of (\ref{riinc}).  We use the redundant formulation to avoid confusion later on.
Then we define $E(\bfr;\rho)$ to be the set of $x\in\reals^d$ such that
\begin{align*}
|x_1| < r_1,\ldots,|x_{d-1}|<r_{d-1},\quad {\rm{dist}}(x+e_d,S^{d-1})<\rho,\quad x_d>-1,
 \end{align*}
 and $F(\bfr;\rho)$ to be the set of $y \in \reals^d$ such that
 \begin{align*}
  |y_1| < \frac{\rho}{r_1}, \ldots,|y_{d-1}| < \frac{\rho}{r_{d-1}}, \quad
{\rm{dist}}(y,S^{d-1})<\rho, \quad y_d < 0.
\end{align*}
Here, $e_d = (0,\ldots,0,1)$.

These pairs of sets are essentially thin neighborhoods of ``dual'' ellipsoids lying on the sphere.  In particular, the ellipsoids share an orientation, the product of corresponding radii is constant, and the ellipsoids lie at opposite poles (though of different spheres).  The pair above is quasi-extremal, as will be shown in \S\ref{basicqexsets}.  Our main theorem states that every $\eps$ quasi-extremal pair $(E,F)$ (after a rotation and translation) is comparable to one of the pairs given above.  This comparability is rather strong and its extent is quantitative in $\eps$.  

Certainly if $R \in O(d)$ is a rotation and $x_0 \in \reals^d$, then the sets 
\begin{align} \label{def:basicqex}
R E(\bfr;\rho) + \{x_0\} \quad RF(\bfr;\rho)+\{x_0\}
\end{align}
are also quasi-extremal.  Throughout, we will refer to such translated and rotated versions of our original pairs, assuming (\ref{rismall}-\ref{rinondeg}), as the {\em basic quasi-extremal pairs}.  

The following theorems will involve constants $C$ and $A$ which depend only on the dimension $d$, and in particular, not on the radii $\bfr,\rho$ or the sets $E,F$.

\begin{thm}\label{mainthm}  For every $\eps>0$ and every $\eps$-quasi-extremal pair
  $(E,F)$, there exists a basic quasi-extremal pair $(\tilde{E},\tilde{F})$ such that
  \begin{align*}
  \scriptt(\tilde{E} \cap E, \tilde{F} \cap F) \geq C^{-1} \eps^{(d+1)/(d-1)}\scriptt(E,F)
  \end{align*}
  and
  \[
  |\tilde{E}| \leq C\eps^{-A}|E| \quad |\tilde{F}| \leq C \eps^{-A}|F|.
  \]
\end{thm}

Hence a quasi-extremal pair $(E,F)$ may be compared to a basic quasi-extremal pair, whose elements are not too much bigger than $E$ and $F$.  We will also show that Theorem \ref{mainthm} implies that a quasi-extremal pair $(E,F)$ may be compared to a basic quasi-extremal pair whose elements are smaller than $E,F$.  

\begin{thm} \label{thmsets2} For every $\eps>0$ and every $\eps$-quasi-extremal pair
  $(E,F)$, there exists a basic quasi-extremal pair $(\tilde{E},\tilde{F})$ such that
\[
  \scriptt(\tilde{E}\cap E,\tilde{F}\cap F) \geq C^{-1}\eps^{A} \scriptt(E,F)
\]
and
\[
  |\tilde{E}| \leq |E| \quad |\tilde{F}| \leq |F|.
\]
\end{thm}

Though Theorem \ref{thmsets2} is perhaps more aesthetically pleasing than Theorem \ref{mainthm} because only one comparison of $(E,F)$ with the basic quasi-extremal pair involves the loss of a power of $\eps$, Theorem \ref{mainthm} is actually the stronger of the two.

Finally, Theorem \ref{thmsets2} is the analogue of the main theorem of \cite{QEx} and, by the arguments in that work, implies the following theorem on $\eps$ quasi-extremal pairs of functions.

\begin{thm} \label{thmfxns} For every $\eps > 0$ and every $\eps$ quasi-extremal pair of non-negative functions $(f,g)$, there exist sets $E,F$ and real numbers $s,t > 0$, such that 
\[
s \chi_E \leq f \quad t\chi_F \leq g;
\]
furthermore, there exists a basic quasi-extremal pair $(\tilde{E},\tilde{F})$ with $|\tilde{E}| \leq |E|$, $|\tilde{F}| \leq |F|$ such that 
\[
s\cdot t\cdot \scriptt(E \cap \tilde{E}, F \cap \tilde{F}) \geq C^{-1}\eps^A \langle Tf,g \rangle.
\]
\end{thm}

Except for the verification of condition (\ref{rinondeg}), our proof is
a more or less straightforward adaptation of \cite{QEx}; where possible,
we will refer to that article for details.  In particular, for the proof of Theorem \ref{thmfxns}, we refer the reader to \cite{QEx}.

\section{Some context}\label{sec:context}

First, we compare the results here to those obtained in \cite{QEx}.  In that work, Christ considered the operator $T_P$, defined by convolution with a measure on the paraboloid $P = \{x \in \reals^d: x_d = |x'|^2\}$;
\[
T_Pf(x) := \int_{\reals^{d-1}} f(x'-t',x_d-|t|^2)\,dt.
\]
He proved that the basic quasi-extremal pairs for $T_P$, also at the $L^p \to L^q$ endpoint $(\frac{d+1}{d},d+1)$, are obtained by applying a rotation $S \in O(d)$ (fixing the $x_d$ component) to the following pairs:  If $x_0,y_0 \in \reals^d$ with $x_0-y_0 \in P$ and $\rho,r_1,\ldots,r_{d-1}$ are {\bf any} positive radii, define $E_P(x_0,y_0,\bfr,\rho)$ to be the set of $x \in \reals^d$ such that
\begin{align*}
|x_i-(x_0)_i| < r_i, \, 1 \leq i \leq d-1, \quad |x_d-(y_0)_d - |x'-{y_0}'|^2| < \rho
\end{align*}
and $F_P(x_0,y_0,\bfr,\rho)$ to be the set of $y \in \reals^d$ such that
\begin{align*}
|y_i-(y_0)_i| < \frac{\rho}{r_i}, \, 1 \leq i \leq d-1, \quad |y_d-(x_0)_d + |y'-{x_0}'|^2| < \rho.
\end{align*}

Thus the families of basic quasi-extremal pairs for the two operators are similar.  The chief difference is in the admissible radii.  For convolution with surface measure on the paraboloid, any collection of radii $\rho,r_1,\ldots,r_{d-1} > 0$ give rise to quasi-extremal pairs.  For convolution with surface measure on the sphere, compactness forces us to take $0 < \rho < 1$ and $\rho < r_i < 1$.  The further condition, (\ref{rinondeg}), however, comes from geometric properties other than compactness.

We will write $E_P(\bfr,\rho) := E_P(0,0,\bfr,\rho)$ and $F_P(\bfr,\rho) := F_P(0,0,\bfr,\rho)$.  

One explanation for the extra rigidity of the quasi-extremals for the sphere is that the paraboloid posses a product structure, while the sphere does not.  One can see a manifestation of this by considering the pair
\[
E:=E_P((1,\ldots,1,\rho,\ldots,\rho);\rho) \quad F:=F_P((1,\ldots,1,\rho,\ldots,\rho);\rho).
\]
The radii are admissible for the paraboloid, and one may think of $E$ and $F$ as exhibiting the product structure $P = P \cap (\reals^k \times \{0\}^{d-1-k}) + P \cap (\{0\}^{k} \times \reals^{d-1-k})$.  In the case of the sphere, however, (\ref{rinondeg}) fails and the radii are not admissible.  

When one considers either the sphere or the paraboloid, two of the basic quasi-extremal pairs of sets are quite well-known.  When the $r_i$ all equal $\rho$, $E$ is a ball of radius $\rho$ and $F$ is a $\rho$-neighborhood of a patch of the hypersurface.  When the $r_i$ all equal $\rho^{1/2}$, $(E,F)$ is essentially the Knapp example; $E \approx [-\rho^{1/2},\rho^{-1/2}]^{d-1}\times[-\rho,\rho]$, $F \approx E+\{y_0\}$ ($y_0 = -e_d$ for $S^{d-1}$ and $0$ for $P$).  Together, these examples are often used to show that $(p,q) = (\frac{d+1}{d},d+1)$ is an endpoint for the Lebesgue bounds for $T$ or $T_P$, as in \cite{Stein}, for instance.  

The operators studied here and in \cite{QEx} are merely examples of a much larger class, called generalized Radon transforms in the recent literature.  

See the articles \cite{PhongStein} by Phong and Stein and more recently \cite{Schlag} by Schlag for some results and an excellent discussion of a larger class of operator defined by integration on curved hypersurfaces.   One motivation for studying these Radon-like transforms comes from partial differential equations.  For example, in 3 dimensions the solution at fixed times to the initial value problem for the wave equation is solved by convolution with surface measure on the sphere.  

At the other dimensional extreme from hypersurfaces, Tao and Wright in \cite{TW} have proved $L^p \to L^q$ bounds near the endpoint for operators defined by integration over one-dimensional curves; Christ has reproved their result in \cite{ChRL} using similar techniques.  The article by Tao and Wright in particular contains an extensive bibliography which may be of interest to the reader. 

Between these extremes, the $L^p \to L^q$ bounds are still largely unknown, but curvature is still important; see \cite{CNSW}.

A more general discussion of the role curvature plays in harmonic analysis may be found, for instance, in the two expository articles, \cite{SteinWainger} and \cite{Iosevich}.

\section{Basic quasi-extremal pairs of sets} \label{basicqexsets}

In this section, we will prove that the basic quasi-extremal pairs of sets are in fact quasi-extremal.

\begin{prop} \label{prop:EFqex}
  Let $1 >\rho>0$ and suppose that $r_1,\ldots,r_{d-1}$ satisfy inequalities
  (\ref{rismall}-\ref{rinondeg}).  Then the pair $(E(\bfr;\rho),F(\bfr;\rho))$, which was defined in \S\ref{prelimrmks}, is quasi-extremal.
\end{prop}

Because the operator $T$ commutes with translations and rotations, we obtain as a corollary:
\begin{cor}
If $R \in O(d)$ is a rotation, $x \in \reals^d$, and $\rho>0$ and $r_1,\ldots,r_{d-1}>0$ satisfy (\ref{rismall}-\ref{rinondeg}), then $(RE(\bfr;\rho)+\{x\}, RF(\bfr;\rho)+\{x\})$ is a quasi-extremal pair.
\end{cor}

\begin{proof}[Proof of proposition \ref{prop:EFqex}]
It suffices to prove the proposition under the additional assumptions that 
\begin{align} \label{ineq:assumptions}
\rho \ll 1 \qquad \rho \ll r_i \ll 1, \, 1 \leq i \leq d-1.
\end{align}
Let $E:=E(\bfr; \rho)$ and $F:=F(\bfr; \rho)$.  Then it is easy to check that 
\begin{align*}
  |E| \sim (\prod_{i=1}^{d-1} r_i) \rho \qquad   |F| \sim (\prod_{i=1}^{d-1} \frac{\rho}{r_i}) \rho,
\end{align*}
so what we must show to verify quasi-extremality of $(E,F)$ is that
\begin{align}\label{ineq:ntsqex}
\scriptt(E,F) \gtrsim \rho^d.
\end{align}

Let $x \in E(c\bfr;c\rho)$, where $c>0$ is sufficiently small (depending only on $d$) for later purposes.  Consider the set 
\[
  \scriptf_0(x) := \{s \in \reals^{d-1}: |s_i-x_i| < c\frac{\rho}{r_i},
  \, 1 \leq i \leq d-1\}.
\]
Then by the smallness of the $r_i$ and $\frac{\rho}{r_i}$,
$|\scriptf_0(x)| \sim \prod_{j=1}^k \frac{\rho}{r_j}$.  We will prove
that 
\begin{align} \label{inc:ntsqex}
s \in \scriptf_0(x) \: \implies \: x-(s,\sqrt{1-|s|^2}) \in F(\mathbf{r},\rho).  
\end{align}
This will imply that $T\chi_F(x) \gtrsim \prod_{j=1}^k \frac{\rho}{r_j}$, from which the estimate (\ref{ineq:ntsqex}) will follow.

First, suppose that $r_{d-1} \leq \rho^{1/2}$.  Let $x\in E(c\bfr;c\rho)$ and $s \in \scriptf_0(x)$.  Then 
\begin{align*}
x_d =  cO(\rho) \quad x_i-s_i = cO(\frac{\rho}{r_i}) \quad x_d-\sqrt{1-|s|^2} = cO(\rho) + \sqrt{1-|x'-s|^2},
\end{align*}
where the second inequality follows from $r_i \leq \rho^{1/2} \leq \frac{\rho}{r_i}$.
Therefore (\ref{inc:ntsqex}) holds.  

Similarly, it is not difficult to show that (\ref{inc:ntsqex}) holds when $r_1 \geq \rho^{1/2}$.  Hence we may assume that there exists an index
$k$, $1 \leq k \leq d-2$ so that
\[
r_k \leq \rho^{1/2} \leq r_{k+1};
\]
recall the monotonicity assumption (\ref{riinc}).

Let $x \in E(c\bfr,c\rho)$ and $s \in \scriptf_0(x)$, and define 
\[
  t = (s_1,\ldots,s_k) \qtrcm and \qtrcm y = (x_{k+1},\ldots,x_{d-1}).
\]

To verify (\ref{inc:ntsqex}), it suffices to show that
\begin{align} \label{ineq:reduction1}
  |x_d-\sqrt{1-|s|^2} + \sqrt{1-|x'-s|^2}| \leq \rho.
\end{align}
We note that $x_d = \sqrt{1-|x'|^2} - 1 + cO(\rho)$ and that by our choice
of $k$,
\begin{align*}
  |x'|^2 = |y|^2 + cO(\rho) \quad |x'-s|^2 = |t|^2 + cO(\rho) \quad |s|^2 = |t|^2 + |y|^2 + cO(\rho).  
\end{align*}
Therefore (\ref{ineq:reduction1}) will follow from 
\begin{align} \label{ineq:reduction2}
  |\sqrt{1-|y|^2} + \sqrt{1-|t|^2} - \sqrt{1-|y|^2-|t|^2} - 1|
   \leq cO(\rho).
\end{align}

By condition (\ref{rinondeg}) on the $r_i$, either $r_1 \geq
\rho^{3/4}$ or $r_{d-1} \leq \rho^{1/4}$.  Assuming the latter, we use
a Taylor series expansion to obtain
\[
  \sqrt{1-|y|^2} = 1-\frac{|y|^2}{2} + cO(\rho)
\]
and
\[
  \sqrt{1-|t|^2-|y|^2} = \sqrt{1-|t|^2}-\frac{|y|^2}{2\sqrt{1-|t|^2}}
  + cO(\rho).
\]
Therefore,
\[
  |\sqrt{1-|y|^2} + \sqrt{1-|t|^2} - \sqrt{1-|y|^2-|t|^2} - 1| = 
  \frac{|y|^2}{2}(\frac{1}{\sqrt{1-|t|^2}}-1) + cO(\rho).
\]
By smallness of the $r_i$ and the $\frac{\rho}{r_i}$, $\sqrt{1-|t|^2}
\sim 1$, so the term on the right is bounded by $c(r_{d-1}^2 \cdot
(\frac{\rho}{r_1})^2 + O(\rho))$.  Inequality (\ref{ineq:reduction2}) follows from another application of condition (\ref{rinondeg}).  

In the other case, $r_1 \geq \rho^{3/4}$, we have that
$\frac{\rho}{r_1} \leq \rho^{1/4}$, and the verification of (\ref{ineq:reduction2}) is the
same as in the previous case, with the roles of $y$ and $t$ reversed.  
\end{proof}


\section{setup for the proof of the main theorem}\label{sec:setup}

Let $(E,F)$ be an $\eps$-quasi-extremal pair of sets.  Using a partition
of unity, we may write $T = \sum_{j=1}^M T_j$, where $T_j$ is
equal to convolution with $a_j \, d\sigma$ and the $T_j$ and $M$ depend on the dimension alone.  Here $d \sigma$ is surface
measure on $S^{d-1}$ and $a_j$ is a smooth function supported on a set
$U_j \subset S^{d-1}$ having diameter $\ll 1$.  By the triangle inequality, $(E,F)$ is $(M^{-1}\eps)$-quasi-extremal for at least one of the
$T_j$.  By means of a rotation, we may assume that $U_j$ is contained in
a small ball centered at $(0,\ldots,0,1)$.  Henceforth, we will write $T
= T_j$, $U = U_j$, and $M^{-1}\eps = \eps$.  Of course, this means, for
instance, that $T$ is no longer self-adjoint.

We define 
\[
  \alpha:= \frac{\scriptt(E,F)}{|E|} \qtrcm \text{and} \qtrcm \beta:=
  \frac{\scriptt(E,F)}{|F|}.
\]
Then $\alpha$ represents the average size of $T\chi_F$ for points in
$E$ and $\beta$ represents the average size of $T^*\chi_E$ for points in
$F$.  See for instance \cite{CCC} for a proof of the following lemma.
\begin{lem}
  For each integer $N \geq 1$, there exists $x_0 \in E$ and measurable
  sets $\Omega_i \subset \reals^{i(d-1)}$, $1 \leq i \leq N$ such that $|\Omega_1| \gtrsim \alpha$, $\Omega_i \subset \Omega_{i-1} \times \reals^{d-1}$, for $2 \leq i \leq N$  and such that whenever $t = (t_1,\ldots,t_i) \in \Omega_i$, $1 \leq i   \leq N$, 
  \begin{align*}
    |\{s \in \reals^{d-1}:(t,s) \in \Omega_{i+1}\}| \gtrsim 
    \begin{cases}
      \beta,\,\text{ if $i<N$ is odd}\\
      \alpha,\,\text{ if $i < N$ is even}
    \end{cases}\\
    x+ \sum_{j = 1}^i (-1)^j(t_j,\sqrt{1-|t|^2}) \in 
    \begin{cases}
      F,\,\text{ if $i$ is odd}\\
      E,\,\text{ if $i$ is even.}
    \end{cases}
  \end{align*}
  Here the implicit constants depend on $N$ and $d$.
\end{lem}

Henceforth we will assume that $N$ is fixed and sufficiently large (say
5 or so) and that $x_0$ and $\Omega_1,\Omega_2,\ldots$ satisfy the
conclusions of the lemma.

\section{The Shape of $\Omega_1$.}\label{sec:Omega1}
Techniques used in this section have previously appeared in the work \cite{QEx} by Christ and are similar to arguments used by Schlag in \cite{Schlag}.

\subsection{Inflation bound.}\label{sec:inflation}
Let $s \in \reals^{d-1}$ and $(t_1,\ldots,t_{d-1}) = \olt \in
\reals^{(d-1)(d-1)}$ with $|s|,|t_i| \ll 1$, $1 \leq i \leq d-1$.  We
define $\Psi^{\natural}:\reals^{(d-1)d} \to \reals^{d(d-1)}$ by
\[
  [\Psi^{\natural}(s,\olt)]_j:=(t_j-s,\sqrt{1-|t_j|^2}-\sqrt{1-|s|^2}),
\]
where we are writing $\bfx \in \reals^{d(d-1)}$ as $\bfx = ([x]_1,\ldots,[x]_{d-1})$ with $[x]_j \in \reals^d$.
We may compute
\[
  \det D \Psi^{\natural} = \det(F_s(t_1),\ldots,F_s(t_{d-1})),
\]
where for $s,t \in \reals^{d-1}$ with $|s|,|t| <1$,
\[
  F_s(t):= \frac{t}{\sqrt{1-|t|^2}}-\frac{s}{\sqrt{1-|s|^2}}.
\]
See the end of this section for a few remarks concerning this
function.

We would like to obtain a lower bound for $|E|$.  For $s \in
\Omega_1$, $|s| \ll 1$, let 
\[
  \scriptf(s):= \{t \in \reals^{d-1}:|t| \ll 1 \qtrcm \text{and}
  \qtrcm (s,t) \in \Omega_2\}.
\]
If we define
\[
  \Omega^{\natural} := \{(s,\olt) \in \reals^{d-1} \times
  \reals^{(d-1)(d-1)}: s \in \Omega_1, \qtrcm \olt \in
  (\scriptf(s))^{d-1}\},
\]
then $\Psi^{\natural}(\Omega^{\natural}) \subset E^{d-1}$.  

Write $\omega_0:=(s,\sqrt{1-|s|^2})$ and $\omega_i:=
(t_i,\sqrt{1-|t_i|^2})$, $1 \leq i \leq d-1$.  Then for each $\bfx =
(x_1,\ldots,x_{d-1}) \in (\reals^d)^{d-1}$, the pre-image of $\bfx$
under $\Psi^{\natural}$ has cardinality
\begin{align*}
  \#\{(s,\olt) &: \Psi^{\natural}(s,\olt) = \bfx\}  \\
  &\leq\# \{(\omega_0,\omega_1,\ldots,\omega_{d-1} \in (S^{d-1})^d :\omega_0
  =  x_i + \omega_i, \qtrcm 1 \leq i \leq d-1\}\\
  & \leq \prod_{j=0}^{d-1}\#\bigcap_{i=0}^{d-1}(S^{d-1}+\{x_i-x_j\}) =
  [\#\bigcap_{i = 0}^{d-1}(S^{d-1}+\{x_i\})]^{d-1},
\end{align*}
where $x_0:=0$.

By Bezout's Theorem (see \cite{Shaf}, for instance), for $\bfx$ lying off of a
measure-zero (indeed, algebraic) subset of $\reals^{d(d-1)}$, the right side is bounded by a
constant depending only on $d$.  An elementary proof of this fact is
also possible; one can take advantage of the fact that the intersection
of two spheres, neither a subset of the other, is either empty or a
lower dimensional sphere.

From the cardinality bound established above, we have the estimate
\[
  |E|^{d-1} \gtrsim \int_{\Omega_1} \int_{\scriptf(s)^{d-1}}
   |\det(F_s(t_1),\ldots,F_s(t_{d-1}))|d\olt \, ds.
\]

\subsection{Approximation by convex sets.}
  We will use an argument from \cite{QEx} to obtain a lower bound for this
term and to describe a typical set $\scriptf(s)$; by symmetry
a similar description will apply to $\Omega_1$.

A set $V \subset \reals^{d-1}$ is {\bf balanced} if $-V=V$.

\begin{lem}
Let $\eta>0$.  Then for any Lebesgue measurable set $A \subset \reals^{d-1}$
with $0 < |A| <\infty$, there exists a bounded, balanced convex set $V
\subset \reals^{d-1}$ such that whenever $V' \subset V$ is a balanced
convex set with $|V'| \leq \frac{1}{2}|V|$, then
\[
  |A \cap (V\backslash V')| \gtrsim \left(\frac{|A|}{|V|}\right)^{\eta}
   |A|.
\]
Moreover,
\[
  \int_{A^{d-1}}|\det(\olu)|d\olu \gtrsim
  |V||A|^{d-1} (\frac{|A|}{|V|})^{\eta(d-1)}
\]
Here the implicit constants depend only on $\eta$ and $d$.
\end{lem}

Note that as a consequence of the lemma, the convex set $V$ satisfies
$|V| \gtrsim |A|$.  

The proof is contained in two somewhat more general lemmas in
\cite{QEx}.  We only present a sketch of the argument here.  The first
part is proved via a stopping-time procedure.  We start with a large
balanced convex set $V$ having size $2^m|A|$ so that $|V \cap A| \geq
\frac{3}{4}|A|$; if $V$ satisfies the conclusion of the lemma, we are
done, otherwise, there is a bad set $V' \subset V$, with which we
replace $V$.  The main trick in this portion of the proof is to use
$\eta$ to show that this procedure terminates before $|V|$ reaches
$|A|$.  The proof of the second part of the lemma uses the defining
property of $V$ and the identity
\[
  |\det(\olu)| = \prod_{i=1}^{d-1} \dist(u_i,V_{i-1}),
\]
where $V_0 = \{0\}$ and $V_i = {\rm{span}}(u_1,\ldots,u_i)$ to bound $\olu$
away from the sets where $\det(\olu)$ vanishes.

We apply the lemma to the set $F_s(\scriptf(s))$, for the moment
leaving $\eta$ undetermined, to obtain a balanced
convex set $V(s) \subset \reals^{d-1}$.  Because near zero, $F_s$ is
a diffeomorphism with bounded differential, we have that
\[
  |V(s)| \gtrsim |F_s(\scriptf(s))| \sim |\scriptf(s)| \sim \beta.
\]

Hence, by making the change of variables $u_i = F_s(t_i)$, $1 \leq i \leq d-1$, we see that
\begin{align*}
  \int_{\Omega^{\natural}}|\det(F_s(t_1),\ldots,F_s(t_{d-1}))|d\olt \, ds
  &\gtrsim \int_{\Omega_1}\int_{F_s(\scriptf(s))^{d-1}}|\det(\olu)|
  d\olu \, ds \\
  & \gtrsim \int_{\Omega_1}|V(s)|^{1-(d-1)\eta}|\scriptf(s)|^{(d-1)(1+\eta)}.
\end{align*}
Assuming that $\eta < \frac{1}{d-1}$, this implies that
\[
  |E|^{d-1} \gtrsim \alpha \beta^d.
\]
From the definitions of $\alpha$ and $\beta$, this immediately yields
the (well-known) bound
\[
  \langle T \chi_E,\chi_F \rangle \lesssim |E|^{d/(d+1)}
  |F|^{d/(d+1)}.
\]

On the other hand, we are assuming that
\[
  \langle T\chi_E,\chi_F\rangle \geq \eps |E|^{d/(d+1)}|F|^{d/(d+1)},
\]
which implies that
\begin{align} \label{ineq:ubE}
  |E|^{d-1} \lesssim \eps^{-(d+1)} \beta^d \alpha.
\end{align}
Since the series of inequalities above imply that
\[
|E|^{d-1} \gtrsim \int_{\Omega_1}
\beta^{(d-1)(1+\eta)}|V(s)|^{1-(d-1)\eta}ds,
\]
the above upper bound on $|E|$ means that for most (in particular,
at least one) of the $s \in \Omega_1$, we must have $|V(s)| \lesssim \eps^{-C} \beta$.  

The above argument (with minor changes) enables us to assume
that for some $\tau \in \reals^n$ with $|\tau| \ll 1$, $F_{\tau}(\Omega_1)$ is
contained in a balanced convex set $V$ of size $\lesssim \eps^{-C}
\alpha$.  By John's theorem, we may assume that $V$ is actually an ellipsoid.

\subsection{Comments on $F_s$}.  
  If we let $g(t) = \sqrt{1-|t|^2}$ for $|t| < 1$, then the mapping
$F_s$ is equal to 
\[
  \nabla g - \nabla g(s).
\]
All of the material in \S\S \ref{sec:Omega1},\ref{sec:scriptfs} is
applicable to any sufficiently nice mapping $g$ with $\nabla g(0) = 0$ and $\det D^2 g(0)
\neq 0$.  In the case of \S \ref{sec:Omega1}, this generality was
pointed out to the author by M. Christ, and in the case of \S
\ref{sec:scriptfs} it can be obtained from his work in \cite{QEx} together
with some additional details in the following section.

For two choices of $g$, $F_s$ has a
particularly nice form.

When $T = T_P$, the operator mentioned in \S\ref{sec:context}, $g(s) = |s|^2$.  In this case, of
course, $F_s(t) = t-s$ for all $s$ and $t$.  Therefore $F_{\tau}(\Omega_1)$ is contained in a
balanced convex set if and only if the vertical projection of $\Omega_1$
to $\reals^{d-1}$ is contained in a convex set which is balanced with
respect to $\tau$.  

When $g(s) = \sqrt{1-|s|^2}$, $\nabla g$ may be thought of as the function which maps a point $\omega \in S^{d-1}$ to the point on $\reals^{d-1} \times 0$ which is collinear with 0 and $\omega$.  This function maps great circles to straight
lines, so $F_{\tau}(\Omega_1)$ is contained in a convex
set if and only if $\Omega_1$ is contained in a geodesically convex
subset of the upper hemisphere.  The point $(\tau,\sqrt{1-|\tau|^2})$
would be contained in this convex set, but the set would not be balanced
about that point in any natural (rotationally invariant) sense.

\section{The shape of $\scriptf(s)$}\label{sec:scriptfs}

\subsection{Slicing Bound.}\label{sec:slicing}
  Let $\Phi$ be defined by 
\[
  \Phi(s,t) = (t-s,\sqrt{1-|t|^2}-\sqrt{1-|s|^2}),
\]
for $s,t \in \reals^{d-1}$ with $|s|,|t|<1$; then $\Phi(\Omega_2)
\subset E$.  Let $B$ be the unit ball in $\reals^{d-1}$, and let $A$ be
a positive definite symmetric linear transformation having norm $\ll 1$.

In this subsection we prove the following lemma.

\begin{lem} \label{lemma:slicing}
If $F_{\tau}(\Omega_1) \subset A(B)$, where $|\tau| \ll 1$, then 
\[
  |E| \gtrsim |\det A|^{-1}\int_{\Omega} |A\,(DF_{\tau}(s))^{-1}F_s(t)|dt\,ds.
\]
\end{lem}

Here we recall that 
\[F_s(t) := \frac{t}{\sqrt{1-|t|^2}} - \frac{s}{\sqrt{1-|s|^2}}.\]

The proof of Lemma \ref{lemma:slicing} is modeled on a proof of the analogous lemma in
\cite{QEx}.  

If $\nu \in \reals^{d-1}$ is a unit vector, and $a \in \nu^{\perp}
\subset \reals^{d-1}$, then we define
\[
  s^{\nu}(a,r) := F_{\tau}^{-1}(A(r\nu+a)).
\]
Thus when $\nu$ is fixed, for each $s \in \Omega_1$, there is a unique
choice of $a,r$ with $|a|,|r| \ll 1$ such that $s = s^{\nu}(a,r)$.  For
the moment, let $\nu$ and $a$ be fixed; we will abuse notation by
writing 
\[
  s(r) = s^{\nu}(a,r).
\]
For $|s| \leq 1$, let $g(s) := \sqrt{1-|s|^2}$, and for $r \in \reals$,
and $u \in \reals^{d-1}$ with $|r|,|u| \ll 1$, define
\[
  \Psi(r,u) := (u,g(u+s(r))-g(s(r))).
\]
Note that when $s = s(r)$ and $u = t-s(r)$, $\Psi(r,u) = \Phi(s,t)$.

Both $(r,u)$ and $\Psi(r,u)$ are elements of $\reals^d$, so one can
compute the Jacobian $\det D\Psi(r,u)$.  The lemma would follow from the
estimate
\begin{align} \label{lbPsi}
  |\Psi(\Omega)| \gtrsim \int_{\omega}|\det D\Psi(r,u)| du \, dr,
\end{align}
for measurable sets $\Omega$ (details are forthcoming).  We will establish the validity of (\ref{lbPsi}) by
showing that $\Psi$ is nearly injective.

Supposing that $\Psi(r,u) = \Psi(r',u')$, one has that $u = u'$ and 
\[
  g(u+s(r)) - g(s(r)) = g(u+s(r')) - g(s(r')) =: f_u(r).
\]
We claim that for $u\neq0$, $r \mapsto f_u(r)$ is $O(1)$-to-1 on $|r|\ll 1$.  By rotating coordinates if necessary, we may assume that $A\nu = e_1$ and set $Aa +\frac{\tau}{\sqrt{1-|\tau|^2}} =: a_0$.  Then 
\begin{align*}
s(r) &= F_0^{-1}(re_1+a_0) = \frac{re_1+a_0}{\sqrt{1+|re_1+a_0|^2}}.
\end{align*}
Using this, one can explicitly compute that $f_u'(r)$ vanishes if and only if a certain polynomial $p_{u,a_0}(r)$ of degree $O(1)$ vanishes.  Since $p_{u,a_0}$ is the zero polynomial if and only if $u=0$, the claim is proved, and we may use the estimate (\ref{lbPsi}).

With $\nu$ and $a$ fixed, we define
\begin{align*}
  \omega^{\nu}_a &:= \{(r,t):(s^{\nu}(a,r),t) \in \Omega_2\}, \\
  \tilde{\omega}^{\nu}_a &:= \{(r,u):(r,u+s^{\nu}(a,r)) \in \omega^{\nu}_a\}.
\end{align*}
So far, we have shown that for each $a$ and $\nu$,
\begin{align*}
  |E| \geq |\Psi(\tilde{\omega}^{\nu}_a)|  
  &=  \int_{\omega^{\nu}_a} |\langle
   A\,[D^2g(s(r))]^{-1}F_{s(r)}(t), \nu \rangle| dt\,dr.
\end{align*}
Now, for each $\nu$, 
\begin{align*}
  |E| &\gtrsim \int_{a \in \nu^{\perp}, |a| \ll 1}
   |\Psi(\tilde{\omega}^{\nu}_a)| da \\ &\gtrsim |\det A|^{-1} \int_{\Omega_2} |\langle
  A\,[D^2g(s)]^{-1}F_s(t),\nu\rangle|dt\,ds,
\end{align*}
where the last inequality follows from the change of variables $s =
s^{\nu}(a,r)$ and the fact that $|\det[DF_{\tau}](s)| \sim 1$.
Averaging with respect to unit vectors $\nu$ completes the proof of the
lemma.

\subsection{Combining the Inflation and Slicing bounds.}\label{sec:combine}

The arguments of this subsection are easy modifications of arguments
due to Christ in \cite{QEx}.  

We have shown the following.  There exist $x_0 \in \reals^d$, $\tau \in
\reals^{d-1}$, $|\tau| \ll 1$, a set $\Omega_1$, and a symmetric,
positive-definite linear transformation $A$, $\|A\| \ll 1$ such that
\begin{align*}
  |\Omega_1| \gtrsim \alpha \quad   \det(A) \lesssim \eps^{-b}\alpha \quad   F_{\tau}(\Omega_1) \subset A(B),\\
  s \in \Omega_1 \implies x_0-(s,\sqrt{1-|s|^2}) \in F
\end{align*}
where $B$ is the unit ball, and $b>0$.  From the main lemma of \S
\ref{sec:slicing}, 
\begin{align*}
  |E| &\gtrsim (\det A)^{-1}\int_{\Omega_1}\int_{\scriptf(s)} |A(DF_{\tau}(s))^{-1}F_s(t)|dt\,ds\\
  & \sim (\det A)^{-2} \int_{\Omega_1} \int_{\tilde{\scriptf}(s)} |w| \, dw\,ds,
\end{align*}
after making the change of variables 
\[
w = A(DF_{\tau}(s))^{-1}F_s(t), \qquad \tilde{\scriptf}(s) = A(DF_{\tau}(s))^{-1}F_s(\scriptf(s))
\]
and using the fact that for $s$, $\tau$ small, \[\det(DF_{s\,\text{or}\,\tau}(t\,\text{or}\,s) )\sim 1.\]

Now for each $s \in \Omega_1$ and each $\rho>0$, either 
\begin{align}\label{ineq:bad1}
\int_{\tilde{\scriptf}(s)} |w|\,dt\,ds \gtrsim \rho|\tilde{\scriptf}(s)| \sim \rho (\det A) |\scriptf(s)| \sim \eps^b\rho \alpha \beta.
\end{align}
or
\begin{align}\label{ineq:good1}
  |\tilde{\scriptf}(s) \cap B(0,\rho)| &\gtrsim |\tilde{\scriptf}(s)|.
\end{align}

Recalling (\ref{ineq:ubE}), if we set $\rho = C \eps^{-C'}(\alpha \beta)^{1/(d-1)}$ ($C,C'$ depending only on $d$), then occurrence of (\ref{ineq:bad1}) over a majority of $\Omega_1$ is impossible; it would contradict the upper bound on $|E|$ which results from quasi-extremality.  Hence (\ref{ineq:good1}) must hold for most $s \in \Omega_1$.  Refining $\Omega_1$ to a subset whose size is still $\gtrsim \alpha$, we
may assume that (\ref{ineq:good1}) holds for each $s \in \Omega_1$.

Unwinding the definition of $\tilde{\scriptf}(s)$, we obtain that 
\[
  |F_s(\scriptf(s)) \cap DF_{\tau}(s)\,A^{-1}B(0,\rho)| \gtrsim
  |F_s(\scriptf(s))| \sim \beta.
\]
We also refine $\scriptf(s)$ to a subset whose size is still $\gtrsim \beta$ and assume that $F_s(\scriptf(s)) \subset DF_{\tau}(s)\,A^{-1}B(0,\rho)$.

\section{Proof of the Main Theorem}

By means of a rotation of $\reals^{d-1}$, we may assume that $A$ is diagonal with
eigenvalues
\[
  r_1 \leq r_2 \leq \ldots \leq r_{d-1}.
\]

By what we have proved so far, one could prove Theorem \ref{mainthm}
without too much trouble (with $\rho$ and the $r_i$ indicating the same
quantities), were it not for the non-degeneracy condition
(\ref{rinondeg}) on the $r_i$.  In the article \cite{QEx}, the special
structure of the paraboloid meant that no such condition was necessary.
The primary work of this section, and one of the main new details of this
article, will be to establish the necessity of that inequality.

As $|E(\bfr;\rho)|$ and $|F(\bfr;\rho)|$ are allowed
to be a factor of $\eps^{-B'}$ larger than $|E|$ and $|F|$, resp., if we
can show that $r_1 \gtrsim \eps^{B'}\rho^{1/2}r_{d-1}$, then by enlarging
the various parameters as needed (for instance increasing the size of
$r_1$ and $\rho$ proportionately to one another) we may achieve the
nondegeneracy condition (\ref{rinondeg}), while maintaining the other
conditions.  We assume, by way of contradiction, that
\begin{align}\label{ineq:bad2}
  r_1 \leq B^{-1}\eps^{B'}\rho^{1/2}r_{d-1},
\end{align}
with $B,B'$ large positive constants, yet to be determined.

In this section, we will need to differentiate between two types of constants, those over which we have control via our assumption (\ref{ineq:bad2}), and those which we cannot substantially influence (e.g.~those appearing in the previous two sections).  We will denote the former by $B,B'$ and the latter by $C,C'$, while allowing the constants to vary from line to line (as is typical in the harmonic analysis literature).  Expressions such as $\lesssim$ will always involve implicit constants of the second type.  Though ``$B$'' will also be used to denote the unit ball, we have otherwise exhausted letters $A-F$, and our meaning will be clear from the context.

\subsection{An alternative description of $\scriptf(s)$.}\label{sec:alt_scriptfs}
To simplify the exposition, we assume that $\tau = 0$; if this were not
the case, it could be achieved by rotating the sphere and enlarging
$\rho$ and the $r_i$ by a bounded factor.

We assume that $s \in F_{0}(\Omega_1) \subset A(B)$, which implies that 
\[
  |s_i| \lesssim r_i, \, 1 \leq i \leq d-1.
\]
From this, $[DF_0(s)]_{i,j} = \delta_{i,j} + O(r_i r_j)$, and $DF_0(s)
\, \rho \cdot A^{-1} = \rho A^{-1} + O(\rho)$.  Therefore, by the smallness of the
$r_i$, $DF_0(s) \, A^{-1}(B(0,\rho)) \subset A^{-1}(B(0,C\rho))$, for some
constant $C$ independent of $\eps$, $\rho$, and the $r_i$.  Henceforth
we will ignore this constant.

We suppose that $t \in \scriptf(s)$; then
\[
  |\frac{t_i}{\sqrt{1-|t|^2}} - \frac{s_i}{\sqrt{1-|s|^2}}| \lesssim
   \frac{\rho}{r_i}, \, 1 \leq i \leq d-1.
\]

By our assumption (\ref{ineq:bad2}), as well as the assumptions $\rho \ll 1$ and $\rho \ll r_1 \leq \ldots \leq r_{d-1} \ll 1$, there exists an index $k$, $1 \leq k < d-1$ such that 
\begin{align} \label{def:k}
r_k \leq \rho^{1/2} \leq r_{k+1}.
\end{align}
We let $t = (t_I,t_{II}) \in \reals^k \times \reals^{d-k-1}$.  We will show
that $t$ can be approximated by 
\[
  (t_I,\sqrt{1-|t_I|^2} \cdot s_{II}),
\]
in the sense that for $i > k$, 
\[
t_i = s_i\sqrt{1-|t_I|^2} + O(\frac{\rho}{r_i}).
\]

If $k+1 \leq i \leq d-1$, then
\[
  |t_i - s_i \sqrt{1-|t_I|^2}| \leq |t_i -
   \frac{s_i\sqrt{1-|t|^2}}{\sqrt{1-|s|^2}}| + |s_i| \cdot
   |\frac{\sqrt{1-|t|^2}}{\sqrt{1-|s|^2}} - \sqrt{1-|t_I|^2}|.
   \]
The first summand on the right is $O(\frac{\rho}{r_i})$, which is acceptable.  The second summand is
\begin{align*}
  \sim  r_i \cdot | |t_I|^2 + |s|^2 - |t|^2 - |t_I|^2|s|^2| 
  & \lesssim  r_i \cdot | |s_{II}|^2(1-|t_I|^2) -
  |t_{II}|^2| + O(\rho),
\end{align*}
because $|s_I|^2 = O(\rho)$.  We continue, ignoring the $O(\rho)$ term.  The right side is
\[
  \lesssim r_i \sum_{j = k+1}^{d-1}|t_j + s_j\sqrt{1-|t_I|^2}|
  \cdot |t_j - s_j\sqrt{1-|t_I|^2}| 
  \leq r_i \sum_{j =
  k+1}^{d-1}r_j|t_j-s_j\sqrt{1-|t_I|^2}|, 
\]
since $\frac{\rho}{r_j} \leq r_j$ when $k+1 \leq j \leq d-1$.  We then have
that
\[
  \sum_{i = k+1}^{d-1}|t_i - s_i\sqrt{1-|t_I|^2} \lesssim \sum_{i =
  k+1}^{d-1} \frac{\rho}{r_i} + \sum_{i = k+1}^{d-1}r_i\sum_{j =
  k+1}^{d-1} r_j|t_j - s_j\sqrt{1-|t_I|^2}|,
\]
which by the monotonicity and smallness of the $r_i$ implies that
\[
  |t_{k+1}-s_{k+1}\sqrt{1-|t_I|^2}| \lesssim \sum_{i=k+1}^{d-1}\frac{\rho}{r_i} \leq (d-k-1)\frac{\rho}{r_{k+1}}.
\]
The other inequalities can be established by induction and the
assumption that $\rho \ll r_i$, $1 \leq i \leq d-1$.

\subsection{A few lower bounds.}\label{sec:lbounds}
  We define
\[
  \tilde{B} := \{\tilde{s} = (s_2,\ldots, s_{d-1}) \in \reals^{d-2}:|s_i| <
  r_i, \, 2 \leq i \leq d-1\};
\]
then
\begin{align*}
  \alpha \lesssim |\Omega_1| &= \int_{\tilde{B}} |\{s_1 \in \reals :
  (s_1,\tilde{s}) \in \Omega_1\}|d\tilde{s} \\
  & = r_2 \cdots r_{d-1} {\rm{avg}}_{\tilde{s} \in \tilde{B}}
  |\{s_1:(s_1,\tilde{s}) \in \Omega_1\}|.
\end{align*}
Dividing both sides by $r_2 \cdots r_{d-1}$ and using the fact that 
\[
  \alpha \lesssim \eps^{-C} \det A = \eps^{-C} r_1 \cdots r_{d-1},
\]
we must then have that
\[
  \eps^C r_1 \lesssim {\rm{avg}}_{\tilde{s} \in \tilde{B}} |\{s_1 :
  (s_1,\tilde{s}) \in \Omega_1\}|.
\]
Say $s_{d-1}$ is good if there exists $(s_2,\ldots,s_{d-2}) \in \reals^{d-3}$ such that $\tilde{s}:=(s_2,\ldots,s_{d-1}) \in \tilde{B}$ and $|\{s_1:(s_1,\tilde{s}) \in \Omega_1\}|>\eps^{C}r_1$.  Then we may choose $C$ and the implicit constant large enough that 
\[
|\{s_{d-1}:s_{d-1} \, \text{ is good}\}| \gtrsim \eps^C r_{d-1}.
\]

Given $s \in \Omega_1$, we know that $F_s(\scriptf(s))$ is contained
in $A^{-1}B(0,\rho)$.  By our assumption (\ref{ineq:bad2}), there
exists $k$, $1 \leq k \leq d-2$ so that 
\[
  r_k \leq \rho^{1/2} \leq r_{k+1}.
\]
Arguing as above, we may assume that $t \in \scriptf(s)$ satisfies
\[
  |[F_s(t)]_j| \gtrsim \eps^C \frac{\rho}{r_j}, \, 1 \leq j \leq k,
\]
where the index $j$ indicates the component, while maintaining the lower bound
\[
  |\scriptf(s)| \gtrsim \beta.
\]
  
These assumptions on $C$ and $\scriptf(s)$ will be in force for the
remainder of this section.

\subsection{More slicing.}  
  Let $\tilde{B}$ be as defined above, and let $\tilde{s} \in \tilde{B}$.  We
define 
\[
  \Omega(\tilde{s}) := \{(s_1,t) \in
  \reals^d: s:=(s_1,\tilde{s}) \in \Omega_1 \, \text{and}\, t \in \scriptf(s)\}
\]
and
\[
  \tilde{E}(\tilde{s}) :=
  \{(t,\sqrt{1-|t|^2})-(s,\sqrt{1-|s|^2}): (s_1,t) \in  \Omega(\tilde{s})\}.  
\]
We have established (in \S\ref{sec:slicing}) that
\[
  \Psi:(s_1,t) \mapsto (t,\sqrt{1-|t|^2})-(s,\sqrt{1-|s|^2})
\]
is sufficiently injective that
\[
  |\tilde{E}(\tilde{s})| \gtrsim
   \int_{\Omega(\tilde{s})} |\det D\Psi(s_1,t)|dt\,ds_1.
\]
Moreover, this is equal to
\[
  \int_{\Omega(\tilde{s})}|[F_s(t)]_1|dt\,ds_1, 
\]
and is bounded from below by (a constant times)
\[
  \eps^C r_1 \cdot \frac{\rho}{r_1} \cdot \beta = \eps^C \rho \beta
\]
whenever $\tilde{s}$ satisfies
\begin{align} \label{ineq:good_s}
\eps^Cr_1 \lesssim |\{s_1:(s_1,\tilde{s}) \in \Omega_1\}|.
\end{align}
By the work of the previous subsection, the set of such admissible $\tilde{s}$'s has size $\gtrsim \eps^C|\tilde{B}|$.

We return to \S\ref{sec:combine} for the definition of $\rho$, which
enables us to conclude that
\[
  |\tilde{E}(\tilde{s})| \gtrsim \eps^C \beta^{d/(d-1)}\alpha^{1/(d-1)},
\]
again, for $\tilde{s}$ satisfying (\ref{ineq:good_s}).

\subsection{A disjointness property}
We will prove in this section that the sets $\tilde{E}(\tilde{s}^{(1)})$ and
$\tilde{E}(\tilde{s}^{(2)})$ are disjoint when $\tilde{s}^{(1)}$ and $\tilde{s}^{(2)}$ are sufficiently far
apart.  When combined with our lower bound on $|\tilde{E}(\tilde{s})|$, this
will give us a stronger lower bound on $|E|$.  As we also have the upper
bound 
\[
  |E| \lesssim \eps^{-C} \alpha^{1/(d-1)}\beta^{d/(d-1)}
\]
from the assumption of quasi-extremality, we will be able to obtain a
contradiction and establish the necessity of (\ref{rinondeg}).

We know from \S\ref{sec:alt_scriptfs} that if $s \in \Omega_1$ and $t \in \scriptf(s)$ implies that
\begin{align*}
|t_i| &\lesssim \frac{\rho}{r_i}, \, 1 \leq i \leq k  \\ 
 |t_i -
  s_i\sqrt{1-|t_I|^2}| &\lesssim \frac{\rho}{r_i}, \, k+1 \leq i \leq
  d-1.
\end{align*}
Therefore $t \in \scriptf(s)$ implies that
\[
  ||t_{II}|^2 - |s_{II}|^2(1-|t_I|^2)| \lesssim \sum_{k+1}^{d-1}
    \frac{\rho}{r_i} (r_i + \frac{\rho}{r_i}) \lesssim \rho
\]

Let $x \in \tilde{E}(\tilde{s})$; write $x=-(s,\sqrt{1-|s|^2}) + (t,\sqrt{1-|t|^2})$, where $t \in \scriptf(s)$.  Then
\begin{align*}
|x_i| &\lesssim \frac{\rho}{r_i}, \, 1 \leq i \leq k, \\
|x_i - s_i(\sqrt{1-|x_I|^2} -  1)| &\lesssim \frac{\rho}{r_i}, \, k+1 \leq i \leq d-1\\
|x_d - \sqrt{1-|s_{II}|^2}(\sqrt{1-|x_I|^2} - 1)| &\lesssim \rho;
\end{align*}
all of these inequalities are independent of $s_1,t$, so they describe arbitrary $x \in \tilde{E}(\tilde{s})$.  Since we are also assuming that  $t \in \scriptf(s)$ implies that $|t_i - s_i| > \eps^C\frac{\rho}{r_i}$, $1 \leq i \leq k$, $x \in \tilde{E}(\tilde{s})$ also satisfies
\begin{align*}
\eps^C\frac{\rho}{r_i} < |x_i|, \, 1 \leq i \leq k,
\end{align*}

Therefore if $x \in \tilde{E}(\tilde{s}^{(1)}) \cap \tilde{E}(\tilde{s}^{(2)})$, then
\begin{align}\label{ineq:s1-s2}
  (1-\sqrt{1-|x_I|^2})|s^{(1)}_i - s^{(2)}_i| \lesssim \frac{\rho}{r_i}, \quad i > k.
\end{align}
We are assuming that $|x_I| > \eps^C \frac{\rho}{r_1}$ (recall that the $r_j$ are monotone), $(1-\sqrt{1-|x_I|^2}) \gtrsim \eps^C\frac{\rho}{r_1}$ and hence by (\ref{ineq:s1-s2}),
\[
  |s^{(1)}_i - s^{(2)}_i| \lesssim \eps^{-C} \frac{r_1^2}{\rho r_i}, i > k.
\]

On the other hand, by the work of \S\ref{sec:lbounds}, we may choose a sequence $\tilde{s}^{(1)},\ldots,
\tilde{s}^{(N)}$ with $N \gg 1$ so that for each $j$, $\tilde{s} = \tilde{s}^{(j)}$ satisfies (\ref{ineq:good_s}), and such that whenever $i \neq j$, $|s^{(i)}-s^{(j)}| \geq \frac{\eps^{C}}{N} r_{d-1}$.
Our assumption (\ref{ineq:bad2}) implies that 
\[
  r_{d-1} \geq B \eps^{-B'} \frac{r_1^2}{\rho r_{d-1}}.
\]
We fix $N$ large enough for later purposes and choose $B,B'$ to be sufficiently large that 
\[
B\frac{\eps^C}{N} \eps^{-B'} \gg \eps^{-C}.
\]
Having done this, we ensure that the sets $\tilde{E}(\tilde{s}^{(j)})$ are pairwise disjoint and thus that the size of their union is 
\[
|\bigcup_{j=1}^N \tilde{E}(\tilde{s}^{(j)})| = \sum_{j=1}^N |\tilde{E}(\tilde{s}^{(j)})| \gtrsim N \eps^C \rho \beta.
\]

The above union is contained in $E$.  Therefore for $N$ sufficiently large, depending only on $C$ and thus ultimately on $d$, the above implies a contradiction to the upper bound (\ref{ineq:ubE}) on $|E|$.

\subsection{Conclusion of proof}
  Now we complete the proof of the main theorem.

From what we have seen so far, there exist $r_1,\ldots, r_{d-1},\rho>0$
satisfying conditions (\ref{rismall}-\ref{rinondeg}) such that
\begin{align*}
  r_1 \cdots r_{d-1} \lesssim \eps^{-C'} \alpha \\
  \frac{\rho}{r_1} \cdots \frac{\rho}{r_{d-1}} \lesssim \eps^{-C'} \beta,
\end{align*}
where
\[
  \rho = C \eps^{-C'} \alpha^{1/(d-1)} \beta^{1/(d-1)},
\]
for some large ($\eps$-independent) constants $C,C'$.

From these, one immediately obtains the upper bounds
\[
  |E(\bfr; \rho)| \lesssim \eps^{-C}|E| \, \text{ and } \, |F(\bfr; \rho)|
   \lesssim \eps^{-C}|F|.
\]

Moreover, by relaxing our assumptions on the size of the $\Omega_i$ to 
\begin{align*}
  |\Omega_1| &\gtrsim \alpha \\
  s \in \Omega_1 &\implies |\scriptf(s)| \gtrsim \beta \\
  (s,t) \in \Omega_2 &\implies |\scriptg(s,t)| := |\{u \in
  \reals^{d-1}:(s,t,u) \in \Omega_3 \}| \gtrsim \alpha,
\end{align*}
we may assume that
\begin{align*}
  \Omega_1 \subset \{s:|[F_0(s)]_i| < r_i, \, 1 \leq i \leq d-1\}\\
  \scriptf(s) \subset \{t:|[F_s(t)]_i| < \frac{\rho}{r_i}, \, 1 \leq i
  \leq d-1\} \\
  \scriptg(s,t) \subset \{u:|[F_t(u)]_i| < r_i, \, 1 \leq i \leq d-1\}.
\end{align*}

From this and our definition of the $\Omega_i$, we have that
\begin{align*}
  &s \in \Omega_1 \implies x_0-(s,\sqrt{1-|s|^2}) \in F \cap (F(\frac{\rho}{\bfr};\rho) +
  \{x_0\}) \\
  &(s,t) \in \Omega_2 \implies x_0-(s,\sqrt{1-|s|^2}) +
  (t,\sqrt{1-|t|^2}) \\ &\qquad \in E \cap (E(\frac{\rho}{\bfr};\rho) +   \{x_0\}) \\
  &(s,t,u) \in \Omega_3 \implies x_0-(s,\sqrt{1-|s|^2}) +
  (t,\sqrt{1-|t|^2})  - (u,\sqrt{1-|u|^2}) \\ &\qquad \in F \cap (F(\frac{\rho}{\bfr};\rho) +
  \{x_0\}).
\end{align*}
From the material in \S\ref{sec:inflation}, one then obtains that
\[
  |\tilde{E}| = |E \cap (E(\bfr,\rho) + \{x_0\})| \gtrsim \beta(\alpha\beta)^{1/(d-1)} \gtrsim \eps^{(d+1)/(d-1)}|E|
\]
(similarly, $|\tilde{F}| \gtrsim \eps^{(d+1)/(d-1)}|F|$).  From this and the lower bound
on $|\scriptg(s,t)|$, we finally have the lower bound
\[
  \scriptt(\tilde{E},\tilde{F}) \gtrsim \alpha |\tilde{E}| \gtrsim \eps^{(d+1)/(d-1)}\scriptt(E,F),
\]
and the theorem is proved.

\section{Proof of Theorem \ref{thmsets2}}
By means of rotations and translations, in proving Theorem \ref{thmsets2}, it
suffices to consider the following situation:  $(E,F)$ is an
$\eps$-quasi-extremal pair, $\bfr,\rho$ satisfy inequalities (\ref{rismall}-\ref{rinondeg}), $\tilde{E} := E \cap
E(\bfr;\rho)$ and $\tilde{F} := F \cap F(\bfr; \rho)$, and
\begin{align*}
  \scriptt(\tilde{E}, \tilde{F}) &\gtrsim C^{-1} \eps^{(d+1)/(d-1)}\scriptt(E,F) \\
  |E(\bfr;\rho)| \lesssim \eps^{-C}|E| \, &\text{ and } \, |F(\bfr;
  \rho)| \lesssim \eps^{-C}|F|.
\end{align*}
We may further assume that $\rho \ll 1$ and $\rho \ll r_i \ll 1$.

Our strategy will be a typical one in harmonic analysis; we will divide
$E(\bfr; \rho)$ and $F(\bfr; \rho)$ into $\lesssim \eps^{-C}$ pairs of
quasi-extremal sets $(E_j,F_j)$ of the correct size ($|E_j| \leq |E|$
and $|F_j| \leq |F|$), and then use the pigeon-hole principle to pick
one pair so that $\scriptt(E_j,F_j) \gtrsim \eps^{C} \scriptt(E,F)$.

We will begin with an initial decomposition of $E(\bfr; \rho)$.  We
choose $\eps^C \leq \lambda < 1$ so that
\[
  \lambda^{d} |E(\bfr; \rho)| \leq |E| \, \text{ and } \, \lambda
  |F(\bfr; \rho)| \leq |F|.
\]
Since $\lambda < 1$, inequalities (\ref{rismall}-\ref{rinondeg}) still hold.  We write
\[
  E(\bfr; \rho) = \bigcup_{i = -\lambda^{-1}}^{\lambda^{-1}} E(\bfr;
  \lambda \rho) + \lambda \rho \cdot i \cdot e_d.
\]

We further decompose $E(\bfr; \lambda \rho)$ as
\[
  E(\bfr; \lambda \rho) = \bigcup_{j = 1}^{O(\lambda^{-(d-1)})}
  B_jE(\lambda\bfr; \lambda \rho),
\]
where $B_j$ is the affine transformation $B_j(x) = R_j(x+e_d)-e_d$, and
$R_j$ is the rotation which takes the point $x^j \in S^{d-1} \cap
(E(\bfr;\rho)+e_d)$ to $e_d$, takes $e_d$ to
$(-x^j_1,\ldots,-x^j_{d-1},x^j_d)$, and fixes all points perpendicular 
to $e_d$ and $x^j$.

Thus
\[
  E(\bfr; \rho) = \bigcup_{i = -\lambda^{-1}}^{\lambda^{-1}} \bigcup_{j
  = 1}^{O(\lambda^{-(d-1)})} B_jE(\lambda \bfr;\lambda \rho) + \lambda
  \rho \cdot i \cdot e_d =: \bigcup_{i = -\lambda^{-1}}^{\lambda^{-1}}
  \bigcup_{j = 1}^{O(\lambda^{-(d-1)})} E_{i,j}
\]

Next, for each $i$ and $j$, we will decompose $F(\bfr; \rho)$ into a
union of sets compatible with $E_{i,j}$; this will be surprisingly easy.  
We note here that the set $F(\lambda \bfr; \lambda \rho)$ has the same
dimensions as $F(\bfr; \rho)$ in the directions perpendicular to $e_d$,
but has thickness $\lambda \rho$ rather than $\rho$ in the $e_d$
direction.

First, we wish to know with which portion of $F(\bfr; \rho)$ an element of
$E_{i,j}$ interacts via convolution with the sphere.

Let $B_{i,j}$ be the affine transformation defined by
\[
  B_{i,j}(x) := B_j(x) + \lambda \rho \cdot i \cdot e_d,
\]
so that $E_{i,j} = B_{i,j}E(\lambda \bfr; \lambda \rho)$.  It now
suffices to determine which elements of $B_{i,j}^{-1}F(\bfr;\rho)$ are a
distance $1$ from elements of $E(\lambda \bfr; \lambda \rho)$.  One can
check, for instance by explicitly computing $B_{i,j}^{-1}(y)$ for $y \in
F(\bfr;\rho)$ that $B_{i,j}^{-1}F(\bfr;\rho) \subset
F(\bfr;C\rho)$ for some constant $C$ depending on the dimension.
Now, by increasing $\rho$ as needed, it is sufficient to determine with
which elements of $F(\bfr;\rho)$ lie a distance 1 from a point in
$E(\lambda \bfr;\lambda \rho)$.  

Write $x \in E(\lambda \bfr;\lambda \rho)$ and $y \in
F(\bfr;C\rho)$ as
\begin{align*}
  x = (x',\sqrt{1-|x'|^2}-1+\delta_1) \qquad
  y = (y',-\sqrt{1-|y'|^2} + \delta_2),
\end{align*}
where $|x_i|<\lambda r_i$, $|y_i|< \frac{\rho}{r_i}$, $|\delta_1|<\lambda \rho$,
and $|\delta_2|<\rho$.  If we assume that $|x-y|=1$, then
\begin{align*}
  1 &= |x'|^2 + |y'|^2 + (x_d-y_d)^2 + O(\lambda \rho) \\
  &= |x'|^2 + |y'|^2 + (\sqrt{1-|x'|^2}-1+\sqrt{1-|y'|^2})^2 \\
  &\hspace{.25cm} - 2\delta_2(\sqrt{1-|x'|^2}-1+\sqrt{1-|y'|^2}) + \delta_2^2
  + O(\lambda \rho).
\end{align*}
Next, one applies the inequality 
\[
  1-\sqrt{1-|x'|^2} \lesssim |x'|^2 \lesssim \lambda r_{d-1}^2,
\]
so
\begin{align*}
  2(1-\sqrt{1-|x'|^2})(1-\sqrt{1-|y'|^2}) \lesssim \lambda r_{d-1}^2
  \frac{\rho^2}{r_1^2} \lesssim \lambda \rho \\
  \delta_2(1-\sqrt{1-|x'|^2}) \lesssim \rho \lambda r_{d-1} \leq \lambda
  \rho.
\end{align*}
Using these inequalities in the series of equalities above,
\[
  -2\delta_2\sqrt{1-|y'|^2} + \delta_2^2 = O(\lambda \rho),
\]
which implies that $|\delta_2| \lesssim \lambda \rho$.

The computations in the paragraph above imply that 
\[
  \scriptt(E(\lambda \bfr;\lambda \rho),F(\bfr;\rho)) =
  \scriptt(E(\lambda \bfr;\lambda \rho),F(\lambda \bfr;C \lambda \rho)).
\]
Therefore
\begin{align*}
  \scriptt(\tilde{E},\tilde{F}) &\leq
  \sum_{i=-\lambda^{-1}}^{\lambda^{-1}} \sum_{j=1}^{O(\lambda^{-(d-1)})}
  \scriptt(E \cap B_{i,j}E(\lambda \bfr;\lambda \rho),F \cap
  F(\bfr;\rho)) \\
  &=   \sum_{i=-\lambda^{-1}}^{\lambda^{-1}} \sum_{j=1}^{O(\lambda^{-(d-1)})}
  \scriptt(E \cap B_{i,j}E(\lambda \bfr;\lambda \rho),F \cap
  B_{i,j}F(\lambda \bfr;C\lambda \rho)) \\
  &\leq   \sum_{i=-\lambda^{-1}}^{\lambda^{-1}}
  \sum_{j=1}^{O(\lambda^{-(d-1)})} 
  \scriptt(E \cap B_{i,j}E(\lambda \bfr;C\lambda \rho),F \cap
  B_{i,j}F(\bfr;C\lambda\rho)).
\end{align*}
By the pigeonhole principle, there exists some choice of $i,j$ so that
\[
  \scriptt(E,F) \lesssim \scriptt(\tilde{E},\tilde{F}) \lesssim
  \eps^{-C}   \scriptt(E \cap B_{i,j}E(\lambda \bfr;C\lambda \rho),F \cap
  B_{i,j}F(\lambda \bfr;C\lambda \rho)),
\]
and the Theorem \ref{thmsets2} is proved.

\bibliographystyle{amsplain}
\bibliography{qexsphere}

\end{document}